\documentclass[english,fleqn,11pt]{article}
\usepackage[utf8]{inputenc}
 \usepackage{hyperref}
\usepackage{a4wide}
 \usepackage{amsmath,amsthm,amscd, amssymb}
 \usepackage{mathtools}
 \usepackage{xcolor}
  \usepackage{graphicx}
 \usepackage{enumitem}
\usepackage{fancyhdr}
 \usepackage{physics} 				
\newcommand{\bZ}{{\mathbb{Z}}}
\newcommand{\bz}{{\mathbf{z}}}
 \DeclareFontEncoding{LS1}{}{}
    \DeclareFontSubstitution{LS1}{stix2}{m}{n}
              \DeclareSymbolFont{symbols2}{LS1}{stixfrak} {m} {n}
		\DeclareMathSymbol{\operp}{\mathbin}{symbols2}{"A8}

 \newtheorem{thm}{Theorem}[section]
	\newtheorem*{thm*}{Theorem}
	\newtheorem*{corr*}{Corollary}
	
	\newtheorem*{lemm*}{Lemma}
	\newtheorem{prop}[thm]{Proposition}
	\newtheorem*{prop*}{Proposition}

\theoremstyle{definition}
	
	\newtheorem*{dfn*}{Definition}
	
	\newtheorem*{exmple*}{Example}

	\newtheorem*{conj*}{Conjecture}

\theoremstyle{remark}

       	\newtheorem{rmk}[thm]{\textit{Remark}}
\def\lset{\{}  	
\def\rset{\}}  	
	\def\set#1{\lset#1\rset} 
	\def\sett#1#2{\lset #1 \mid  #2 \rset}  
	\renewcommand\emptyset{\varnothing} 	
\newcommand\stab[2]{{#1}_{#2}}	
\def\id{\text{\rm id}} 
\def\chr#1{\operatorname{char}({#1})}

\begin{document}

 \title{On the center of the Clifford Algebra}
\author{Chris Peters \& Hans Sterk}
\date{October  2026}
\maketitle
\begin{abstract}
The center of a Clifford algebra associated to a non-degenerate quadratic form over a field is described in the existing literature, but an elementary proof of the result seems lacking. The aim of this paper is to provide such an elementary proof.
  \end{abstract} 
\section{Introduction}  Let $R$ be a commutative ring with unit $1$.   Recall that a quadratic form on an $R$-module $V$
is a function $q:V\to R$ with  $q(rx)= r^2 q(x)$ and such that the form 
 $b_q(x,y)= q(x+y)- q(x) -q(y)$ is a symmetric bilinear form on $V$, also called its polar form. 
 
 The Clifford algebra $\mathsf C(q)$ associated to $q$ can be defined for quadratic forms over any $R$-module (see e.g. \cite[Ch. 2]{quadformen}), but we
 shall be interested in  the case of  finite-dimensional   vector spaces $(V,q)$ over any field $k$ with $q$ non-degenerate.
 These will be called quadratic vector spaces over $k$.

 In that situation and if moreover $\chr k\not=2$, there is a  $q$-orthogonal basis for $V$
 making   $b_q$   diagonal. The form $b_q$ (and by convention also $q$) is non-degenerate 
  if  and only if all diagonal elements are non-zero.
 In case $\chr k =2$ any non-degenerate quadratic form on $V$   has even rank $2n$  and in some  symplectic basis 
such a form can be written    $q=\sum_{i=1}^n  x_i x_{n+i} + \sum_{j=1}^{2n}  a_j x_j^2,$ for some $a_j\in k$.   Its Arf invariant  
  $$
  \mathsf{a}(q):=\sum a_ia_{i+n} \bmod   \wp  (k)   ,  \quad \wp(a)=a+a^2,  a\in k,
$$
  determines the isometry class of $q$.

 If $k$ is a field,   $\mathsf C(q)$
 is the quotient of the tensor  algebra  $T(V)$  on $V$
by  the two-sided  ideal  $I(q)$ generated by $x\otimes x - q(x)\cdot \bf 1$, $x\in V$. The grading on $T(V)$ induces a $\bZ/2\bZ$-grading
on the Clifford algebra so that $\mathsf C(q)=\mathsf C(q^0)\oplus \mathsf C(q^1)$, and since  $\mathsf C^i(q)\cdot \mathsf C^j(q)= 
\mathsf C^{i+j}(q)$, the submodule $\mathsf C^0(q)$ is in fact a sub-algebra.

Usually the computation of the centers of the Clifford algebra and the even Clifford algebra are performed
by considering first the cases where $n=1$, respectively $n=2$, and then using that 
$\mathsf C(q_1 \operp q_2)= \mathsf  C (q_1) \widehat{\otimes} \mathsf C(q_2)$, where $ \widehat{\otimes} $ is the graded tensor product.
See for instance \cite[Ch 9.2,9.3]{scharlau}. The result is as follows.
%
%

\begin{prop}[$\chr k\not =2$]   \label{prop1} Let $(V,q)$ be  a quadratic   vector  space over  $k$  with
 an orthogonal basis $\set{e_1,\dots,e_n}$. Setting
 $e= e_1\cdots e_n\in \mathsf{C}(q)$, we have:
\begin{enumerate} \renewcommand{\labelenumi}{\alph{enumi})}
\item  $\text{ The center of }\mathsf{C}(q) = \begin{cases}
k & \text{ if  $n$ is even}\\
k+k\cdot e& \text{ if  $n$ is odd}\\
\end{cases}
$
\item 
$ \text{The center of }\mathsf{C}^0(q) = \begin{cases}
k & \text{ if  $n$ is odd}\\
k+k\cdot e& \text{ if  $n$ is even}\\
\end{cases}
$
\end{enumerate}

\end{prop}

\begin{prop}[$\chr k  =2$]   \label{prop2}   Let $(V,q)$ be a quadratic vector space over $k$, and let  $\set{e_1,\dots,e_{2n}}$  be a  symplectic basis such that
$q=\sum_{j=1}^n  x_i x_{n+i} + \sum_{j=1}^{2n}  a_j x_j^2$. 
Let $\bz = e_1e_{n+1} +\cdots + e_n e_{2n} \in \mathsf{C}^0(q)$. Then

\begin{enumerate}\renewcommand{\labelenumi}{\alph{enumi})}
\item  the center of $\mathsf{C}(q)$ is $k$,
\item  the center of $\mathsf{C}^0(q)$ equals $k+ k\bz$,
\item $\bz^2+\bz  \equiv \mathsf{a}(q)  \bmod \wp(k)$, where we recall that  $\mathsf{a}(q) \bmod \wp(k)$ is the Arf invariant of $q$.
\end{enumerate}
\end{prop} 

The goal of the present  note is to give a short  elementary  proof of these  results avoiding using the graded tensor product.

\begin{rmk} The results above have rough analogs  for free $R$-modules of finite rank (again $R$ a unital commutative ring) equipped with a unimodular 
bilinear form,  and in case  $2\not\in R^\times$, for quadratic forms $\sum_{i<j}q_{ij} x_ix_j$ that are regular or semi-regular.   See \cite[7.10]{quadformen} for details.
\end{rmk}

 \section{Proof of Proposition \ref{prop1}}
 It suffices to show that
\[
\stab{\mathsf{C}(q)}{\mathsf{C}^0(q)}= \sett{x \in \mathsf{C}(q)}{ xy =yx \text{ for all } y\in \mathsf{C}^0(q)}= k+ ke.
\] 
 Indeed, noting that $e\in \mathsf{C}^0(q)$ precisely if $n$ is even, 
we obtain from $\stab{\mathsf{C}(q)}{\mathsf{C}^0(q)}=k+ke$ 
that the center of $\mathsf{C}^0 (q)$ is $k+ke$ resp.\ $k$ if $n$ is even resp.\ odd. Since $e_i e = (-1)^{n-1} e e_i$ we conclude that the center of $\mathsf{C}(q)$ is $k+ke$ resp.\ $k$ if $n$ is odd resp.\ even.

 To proceed,
let  $u\in\stab{\mathsf{C}(q)}{\mathsf{C}^0(q)}$. Since $\mathsf {C}^0 (q)$ is generated by the elements $e_i e_j$, $i<j$, it suffices to analyze the equations $e_i e_j u = e_i e_j u$. Now note that, writing $J= \{ j_1 , \ldots , j_r\}\subset \{ 1, \ldots , n\}$ and $e_J = e_{j_1} \cdots e_{j_r}$, we have $e_i e_j e_J = \pm e_J e_i e_j$ where the sign is negative if and only if precisely one of the two indices $i,j$ belongs to $J$. 
Given $i\neq j$, write  $u = \sum_J a_J e_J + \sum_K b_K e_K$ where the $J$  either contain both  $i,j$  or neither of them, 
and the $K$ are such that precisely one of $i,j$ is in $K$.
By assumption  $e_i e_j \cdot u = u  \cdot e_i e_j$. On the other hand,
$$
\begin{array}{ll}
e_i e_j \cdot u = e_ie_j(\sum_J a_J e_J + \sum_K b_K e_K) & = (\sum_J a_J e_J - \sum_K b_K e_K)e_i e_j\\
& = u\cdot e_i e_j - 2(\sum_K b_K e_K)e_i e_j.
\end{array}
$$
It follows that  $v=(\sum_K b_K e_K)e_i e_j =0$ (since $\text{char} \neq 2$). 
We claim that  the    $k$-linear endomorphisms $E_\alpha: \mathsf{C}(q)\to \mathsf{C}(q)$ given by right  multiplication with $e_\alpha$
are  isomorphisms. This follows since $E_\alpha^2= q(e_\alpha)\cdot \id$ is invertible ($q(e_\alpha)\neq 0$ since $q$ is non-degenerate). 
Hence $E_iE_j$ is invertible, and so, because  $v\in \ker(E_iE_j)=\set{0}$,  all coefficients $b_K$ of $v$ vanish. In other words,
 all coefficients of terms of $u$ that contain exactly one $i$ or $j$ are $0$. Since $i,j$ are arbitrary, it follows that all coefficients of $u$ are $0$ except possibly the constant term and the term which is a multiple of $e=e_1 e_2 \cdots e_n$. 

 \section{Proof of Proposition \ref{prop2}}
In this case in the Clifford algebra one has the relations
\[
\begin{array}{rcll}
e_i^2  & = & a_i ,                               &  i =1,\dots,2n ,\\
e_i e_j & = & e_je_i,                             &  i,j =1,\dots,n,  \text{ or }  i,j =n+1,\dots,2n, \\
e_i e_{n+j} + e_{n+j} e_i&=&\delta_{ij}   , &    i,j =1,\dots,n .
\end{array} 
\]
\textbf{\emph{Proof of a) and b).}}
Again it suffices to show that $\stab{\mathsf{C}(q)}{{\mathsf{C}^0 (q)}} = k + k\bz$, as $\bz \in {\mathsf{C}^0 (q)}$ and $\bz \cdot e_i = e_i \cdot \bz + e_i$, $i=1, \ldots , n$, then imply that the center of $\mathsf{C}(q)$ is $k$ and the center of $\mathsf{C}^0 (q)=k+k\bz$. 

Let $x\in \stab{\mathsf{C}(q)}{\mathsf{C}^0(q)}$. Since the field $k$ is in the center we may suppose $x= \sum \alpha_{IJ} e_I e_J$ where $I=\set{i_1,\dots,i_k}\subset \set{1,\dots,n}$ with  $i_1< \cdots<i_k$ and 
 $J=\set{j_1,\dots,j_\ell}\subset \set{n+1,\dots,2n}$ with $j_1 < \cdots < j_s$ are not both empty, 
 and where $e_I e_J= e_{i_1}\cdots e_{i_k} e_{j_1}\cdots e_{j_\ell}$.
 
 The proof relies on analyzing various commutators with suitably chosen terms $e_i e_j$ ($1\leq i,j \leq 2n$) of degree 2. The computations carried out below show that a commutator $\left[ e_I e_J, e_i e_j\right]$ may be non-zero with one or more terms of degree $0$ or $2$ less than the degree of $e_I e_J$. These extra terms are used to analyze the coefficients $\alpha_{IJ}$ and to conclude that $x$ reduces to a multiple of $\bz$. 
 
 In the first step we show that $x$ only has even degree terms and that in every nonzero term $\alpha_{IJ}e_I e_J$ with $I=\{ i_1 , \ldots , i_r\}$ we have $J=\{ n+i_1 , \ldots , n+i_r\}$. Next we show that the degree 2 part of $x$ is a scalar multiple of $\bz$. Finally we show inductively that $x$ has no degree $2m$ part for $m=2,3\ldots$. 

\medskip
 \noindent
\textbf{Step 1: The element $x$ has only even degree terms}.\\
Consider a pair of indices $i\in \{ 1, \ldots , n\}$, $n+i\in \{ n+1, \ldots , 2n\}$. For a term $e_I e_J$ we distinguish four cases:
\begin{itemize}
    \item 
 $i\not\in I$, $n+i \not\in J$.  In this case $e_i$ and $e_{n+i}$ commute with $e_I$ and $e_J$ so that 
 $$
 [e_Ie_J, e_i e_{n+i}] =e_Ie_J \cdot e_i e_{n+i} -e_i e_{n+i}\cdot e_I e_J =0.
 $$
 \item 
 If $i=i_k\in I=\{ i_1, i_2, \ldots , i_r\}$ and $n+i=j_m\in J=\{ j_1, j_2, \ldots , j_s\}$, then 
 $$
 [e_Ie_J, e_i e_{n+i}] =0,
 $$
 since
 $$
 \begin{array}{ll}
 e_Ie_J\cdot e_{i_k} e_{j_m}& = e_{i_1}\cdots e_{i_r} e_{j_1} \cdots e_{j_s} \cdot e_{i_k} e_{j_m}\\
 & = e_{i_k} e_{j_m}\cdot e_{i_1} \cdots\widehat{e}_{i_k} \cdots e_{i_r} \cdot e_{j_1}\cdots \widehat{e}_{j_m}\cdots e_{j_s}\cdot e_{i_k} e_{j_m}\\
 & = e_{i_k} e_{j_m}\cdot e_{i_1} \cdots e_{i_r} \cdot e_{j_1}\cdots e_{j_s}\\
 & = e_{i_k} e_{j_m}\cdot e_I e_J.
 \end{array}
 $$
 \item 
 If $i=i_k\in I=\{ i_1 , i_2, \ldots , i_r\}$ and $n+i\not\in J=\{ j_1 , \ldots , j_s\}$, we have
 $$
 \left[ e_I e_J, e_{i_k} e_{n+ i_k}\right] = e_I e_J .
 $$
 This can be seen as follows. Since $e_J$ commutes with $e_{i_k} e_{n+ i_k}$ and 
 $$
 e_I \cdot e_{i_k} e_{n+ i_k} = e_{i_k} \cdot e_{i_1}  \cdots (e_{n+i_k} e_{i_k} +1) \cdots e_{i_r} = e_{i_k} e_{n+i_k} \cdot e_I + e_I 
 $$
 we have 
 $$
\left[ e_I e_J, e_{i_k} e_{n+ i_k}\right]= \left[ e_I, e_{i_k} e_{n+i_k}\right] e_J = e_I e_J.
 $$
\item 
Similarly, if $i\not\in I = \{ i_1 , \ldots , i_r\}$ and $n+i=j_m\in J=\{ j_1, j_2, \ldots , j_s\}$, then $e_I$ commutes with $e_i e_{n+i}$ and $e_J \cdot e_i e_{j_m}=e_{j_1} \cdots (e_i e_{j_m} +1) \cdots e_{j_s} e_{j_m} = e_i e_{j_m}\cdot e_J + e_J$ so that 
$$
[e_I e_J, e_i e_{n+i}] = e_I \left[ e_J, e_i e_{n+i}\right] = e_I e_J.
$$
\end{itemize}
So for $x= \sum \alpha_{IJ} e_I e_J$ we have:
 $$
\left[ \sum \alpha_{IJ} e_Ie_J,e_i e_{n+i}\right]=    \sum_{i\in I, n+i\not\in J}  \alpha_{IJ} e_I e_J  + \sum_{i\not\in I, n+i\in J}  \alpha_{IJ} e_I e_J .
 $$ 
 Consequently, for $i\in I$,  $n+i\not\in J$, and for $i\not\in I$ and $n+i\in J$ the coefficients $\alpha_{ I J}$ must be zero.
We conclude that every non-zero term $\alpha_{IJ} e_Ie_J$ in $x$ satisfies $i\in I \Leftrightarrow n+i\in J$. In particular, $|I|=|J|$ and $x$ does not contain odd degree terms. 

\medskip
\noindent
\textbf{Step 2: The degree 2 part of $x$ is a multiple of $\bz$}.\\
Write $x$ in the form 
$$
x=\sum_I \alpha_I e_I e_{I+n}  \text{ with  } I: i_1 < \cdots < i_k , \, I+n: n+i_1 < \cdots < n+i_k.
$$
Note that the degree 2 part is $x_2 = \alpha_1 e_1 e_{n+1} + \alpha_2 e_2 e_{n+2} + \cdots + \alpha_n e_n e_{2n}$. We argue that all coefficients of the degree 2 terms are equal. To do so, we consider commutators of the form $\left[ x, e_k e_{n+\ell}\right]$ with $k,\ell \in \{ 1, \ldots , n\}$ and $k\neq \ell$. 

 Observe that $e_k e_{n+\ell}$ commutes with every $e_i e_{n+i}$, $i\neq k , \ell$.  For $i=k$ and $i=\ell$ we have, respectively,
$$
\begin{array}{l}
e_k e_{n+k}\cdot e_k e_{n+\ell}  = e_k e_{n+\ell} \cdot (e_k e_{n+k}+1) 
 = e_k e_{n+\ell} \cdot e_k e_{n+k} + e_k e_{n+\ell}, \\
 e_{\ell} e_{n+\ell} \cdot e_k e_{n+\ell} = e_k (e_{n+\ell} e_{\ell} +1) e_{n+\ell} = e_k e_{n+\ell}\cdot e_{\ell} e_{n+\ell} + e_k e_{n+\ell}.
\end{array}
$$
Hence 
$$
\left[ e_i e_{n+i}, e_k e_{n+\ell}\right] = 
\left\{ 
\begin{array}{ll}
0 & \text{if } i\neq k, \ell\\
e_k e_{n+\ell} & \text{if } i=k \text{ or } i=\ell.
\end{array}
\right.
$$
So taking the commutator of $x_2$ with $e_k e_{n+\ell}$ gives  
$$
[x_2 , e_k e_{n+\ell}]=\alpha_k e_k e_{n+\ell} + \alpha_{\ell} e_k e_{n+\ell}=(\alpha_k + \alpha_{\ell}) e_k e_{n+\ell}.
$$
We cannot yet conclude that $\alpha_k =\alpha_{\ell}$ since commutators of $e_k e_{n+\ell}$ with degree 4 terms in $x$ may interfere (the list of commutators below also shows that no commutator with a higher degree term produces a degree 2 term). So we look at what happens to $x_4$, the degree 4 part of $x$. We may assume that $k< \ell$. 
Of the degree 4 terms of $x$, the only term whose commutator with $e_k e_{n+\ell}$ produces degree 2 terms is $\alpha_{k\ell} e_k e_{\ell} e_{n+k} e_{n+\ell}$. In fact,
$$
\left[ e_k e_{\ell} e_{n+k}e_{n+\ell} , e_k e_{n+\ell} \right] =
a_{n+\ell}\cdot  e_k e_{\ell} + a_k\cdot e_{n+k}e_{n+\ell} .
$$
This is the last case in the list
$$
[e_i e_j e_{n+i} e_{n+j},e_k e_{n+\ell}]=
\left\{ 
\begin{array}{ll}
0 & \text{if } i\neq k, \ell, j\neq k,\ell\\
e_i e_k e_{n+i} e_{n+\ell}   & \text{if } i<k, j=k\\
e_k e_j e_{n+\ell}e_{n+j} & \text{if } i=k, j\neq \ell\\
 e_k e_j e_{n+\ell} e_{n+j} & \text{if } i=\ell, j> \ell\\
e_k e_i e_{n+i} e_{n+\ell} & \text{if } i\neq k, j=\ell\\
a_{n+\ell} \cdot e_k e_{\ell} + a_k\cdot e_{n+k} e_{n+\ell} & \text{if } i=k, j=\ell
\end{array} 
\right.
$$
We check the various relations. For $i=k, j=\ell$
$$
\begin{array}{ll}
e_k e_{\ell} e_{n+k}e_{n+\ell} \cdot e_k e_{n+\ell} & = 
e_k (e_{n+\ell} e_{\ell} +1) (e_k e_{n+k} +1) e_{n+\ell} \\
& = e_k e_{n+\ell} \cdot e_k e_{\ell} e_{n+k}e_{n+\ell} + e_k e_{n+\ell}e_{\ell} e_{n+\ell} + a_k\cdot e_{n+k} e_{n+\ell} + e_k e_{n+\ell}\\
&=e_k e_{n+\ell} \cdot e_k e_{\ell} e_{n+k}e_{n+\ell} + e_k (e_{\ell} e_{n+\ell} +1) e_{n+\ell} + a_k\cdot e_{n+k} e_{n+\ell} + e_k e_{n+\ell}\\
& =e_k e_{n+\ell} \cdot e_k e_{\ell} e_{n+k}e_{n+\ell} + a_{n+\ell} e_k e_{\ell} + e_k e_{n+\ell} + a_k\cdot e_{n+k} e_{n+\ell} + e_k e_{n+\ell}\\
& = e_k e_{n+\ell} \cdot e_k e_{\ell} e_{n+k}e_{n+\ell} + a_{n+\ell} e_k e_{\ell} + a_k\cdot e_{n+k} e_{n+\ell}.
\end{array}
$$
For $i<k$, $j=k$ we have
$$
\begin{array}{ll}
e_i e_k e_{n+i} e_{n+k} \cdot e_k e_{n+\ell} & = e_i e_{n+i} e_k (e_k e_{n+k} +1) e_{n+\ell} \\
& = e_i e_{n+i} e_k^2 e_{n+k} e_{n+\ell} + e_i e_{n+i} e_k e_{n+\ell} \\
& = e_k e_{n+\ell}\cdot e_i e_k e_{n+i}e_{n+k} + e_i e_k e_{n+i} e_{n+\ell}.
\end{array}
$$
The remaining cases where $|\{ i,j\} \cap \{ k,\ell\} |=1$ follow from a similar computation. 

So only $x_4$ and $x_2$ contribute to the degree 2 part of $\left[ x, e_k e_{n+\ell}\right]$ yielding: 
\begin{equation}\label{degree2}
(\alpha_k + \alpha_{\ell}) e_k e_{n+\ell} + \alpha_{k\ell} (a_{n+\ell}\cdot e_k e_{\ell} + a_k\cdot e_{n+k} e_{n+\ell}).
\end{equation}
Hence $\alpha_k = \alpha_{\ell}$ for $k< \ell$. So $x_2$ is a multiple of $\bz=e_1 e_{n+1} + \cdots + e_n e_{2n}$. 

Note that we cannot yet conclude that the coefficient $\alpha_{k\ell}$ is zero (see equation \ref{degree2}), since $a_k =q(e_k)$ and $a_{n+\ell}=q(e_{n+\ell})$ may both be zero.

\medskip\noindent
\textbf{Step 3: The coefficients in degree 4 are zero}.\\
A degree 4 term in $x$ is of the form $\alpha_{ij} e_i e_j e_{n+i} e_{n+j}$ with $1\leq i<j \leq n$. Here we consider commutators with $e_k e_{\ell}$ (where $1\leq k < \ell \leq n$). Such commutators only produce degree 2 terms if $i=k$ and $j=\ell$:
$$
\left[ e_i e_j e_{n+i} e_{n+j}, e_k e_{\ell} \right] =
\left\{ 
\begin{array}{ll}
0 & \text{if } \{ i,j\} \cap \{ k,\ell \} =\emptyset\\
e_i e_k e_{\ell} e_{n+i} & \text{if } i< j=k\\
e_k e_j e_{\ell} e_{n+j}\text{ or } e_k e_{\ell} e_j e_{n+j} & \text{if } i=k, j\neq \ell\\
e_k e_{\ell} e_j e_{n+j}  & \text{if } i=\ell , j>\ell\\
a_k\cdot e_{\ell} e_{n+k} + a_{\ell}\cdot e_k e_{n+\ell} + e_k e_{\ell} & \text{if } i=k, j=\ell
\end{array}
\right.
$$
Using $e_m e_{n+m} \cdot e_m = e_m (e_m e_{n+m}+1) = e_m \cdot e_m e_{n+m} + e_m$, the last equality follows from 
$$
\begin{array}{ll}
e_k e_{\ell}e_{n+k} e_{n+\ell}\cdot e_k e_{\ell} & = (e_k \cdot e_k e_{n+k}+e_k)(e_{\ell}\cdot e_{\ell}e_{n+\ell} + e_{\ell})\\
& = e_k e_{\ell}\cdot e_k e_{\ell} e_{n+k}e_{n+\ell} + a_k\cdot e_{\ell} e_{n+k} + a_{\ell}\cdot e_k e_{n+\ell} + e_k e_{\ell}.
\end{array}
$$
If $i=k$, $j\neq \ell$, we have
$$
\begin{array}{ll}
e_k e_j e_{n+k} e_{n+j}\cdot e_k e_{\ell} & = (e_k e_k e_{n+k}+e_k) e_{\ell} e_j e_{n+j} \\
& = e_k e_{\ell} \cdot e_k e_j e_{n+k}e_{n+j} + e_k e_j e_{\ell} e_{n+j} .
\end{array}
$$
The other cases with $|\{ i,j\} \cap \{ k.\ell\}|=1$ are similar.

So only the commutator of $e_k e_{\ell}$ with the degree 4 term $\alpha_{k\ell} e_k e_{\ell} e_{n+k} e_{n+\ell}$ produces degree 2 terms: $\alpha_{k\ell}(a_k\cdot e_{\ell} e_{n+k} + a_{\ell}\cdot e_k e_{n+\ell} + e_k e_{\ell})$. From $\alpha_{k\ell} e_k e_{\ell} =0$ we infer that $\alpha_{k\ell}=0$.   

\medskip
\noindent
\textbf{Step 4: The coefficients in degrees 6 and higher are zero}.\\
A computation as in the degree 4 case shows that a commutator with $e_k e_{\ell}$ of a term $\alpha_{IJ} e_I e_J$ in degree 6 produces a term in degree 4 precisely if $k,\ell \in I$. This implies that all coefficients $\alpha_I$ with $k,\ell \in I$ are $0$. Varying $k$ and $\ell$ we obtain $x_6 =0$. 

A similar reasoning shows that all terms in degrees $2m$, $m\geq 3$, are zero.

\bigskip
\noindent
\textbf{\emph{Proof of c.}}
 This follows from 
 $$
 \begin{array}{ll}
 \bz^2+\bz & =(e_1 e_{n+1} + \cdots + e_n e_{2n})^2 + e_1 e_{n+1} + \cdots + e_n e_{2n}\\
 & = (e_1 e_{n+1})^2 + \cdots + (e_n e_{2n})^2 + e_1 e_{n+1} + \cdots + e_n e_{2n} \\
 & = e_1 (e_1 e_{n+1} +1)e_{n+1} + \cdots + e_n (e_n e_{2n}+1)e_{2n} + e_1 e_{n+1} + \cdots + e_n e_{2n}\\
 & = a_1 a_{n+1} + \cdots +a_n a_{2n}.
 \end{array}
 $$

\bibliographystyle{acm}

 \end{document}